\documentclass[11pt]{amsart}

\usepackage{amsmath}

\usepackage{bbm}
\usepackage{enumerate}
\usepackage{paralist}
\usepackage{amssymb}

\usepackage{stmaryrd}
\usepackage{mathrsfs}

\newtheorem{theorem}{Theorem}

\newtheorem{lemma}[theorem]{Lemma}

\theoremstyle{definition}
\newtheorem{definition}[theorem]{Definition}
\theoremstyle{remark}

\newcommand \NN{\mathbbm{N}}
\newcommand \UU{\mathbbm{U}}
\newcommand \RR{\mathbbm{R}}

\newcommand \SI{S_{\infty}}
\newcommand \Iso{{\rm Iso }}
\newcommand \Aut{{\rm Aut }}

\newcommand \qftp[1]{{\rm qftp}( #1) }

\newcommand \spa{{\rm span}}
\newcommand \MA{{\rm MALG}}

\begin{document}

\title[Homogeneous metric structures]{Consequences of the existence of ample generics and automorphism groups of homogeneous metric structures}
\author{Maciej Malicki}

\address{Department of Mathematics and Mathematical Economics, Warsaw School of Economics, al. Niepodleglosci 162, 02-554,Warsaw, Poland}
\email{mamalicki@gmail.com}
\date{Nov 11, 2014}
 
\begin{abstract}
We define a criterion  for a homogeneous, complete metric structure $X$ that implies that the automorphism group $\Aut(X)$ satisfies all the main consequences of the existence of ample generics: it has the automatic continuity property, the small index property, and uncountable cofinality for non-open subgroups. Then we verify it for the Urysohn space $\UU$, the Lebesgue probability measure algebra $\MA$, and the Hilbert space $\ell_2$, thus proving that $\Iso(\UU)$, $\Aut(\MA)$, $U(\ell_2)$, and $O(\ell_2)$ share these properties. We also formulate a condition for $X$ which implies that every homomorphism of $\Aut(X)$ into a separable group $K$ with a left-invariant, complete metric, is trivial, and we verify it for $\UU$, and $\ell_2$.
\end{abstract}

\maketitle
\section{Introduction}

A Polish (that is, separable and completely metrizable) topological group $G$ has ample generics if the diagonal action of $G$ on $G^n$ by conjugation has a comeager orbit for every $n \in \NN$. This notion was first introduced by Hodges, Hodkinson, Lascar and Shelah \cite{HoHo} to study the small index property for automorphism groups of certain countable structures (where a countable structure $M$ is regarded as a discrete space, and $\Aut(M)$ is equipped with the product topology.) Then Kechris and Rosendal \cite{KeRo} fully characterized (locally finite, ultrahomogeneous) countable structures whose automorphism groups have ample generics. They also proved that the existence of ample generics has very strong consequences. Every Polish group $G$ with ample generics has the automatic continuity property (that is, every homomorphism from $G$ into a separable group is continuous), the small index property (that is, every subgroup $H \leq G$ with $[G:H]<2^\omega$ is open), and uncountable cofinality for non-open subgroups (that is, there are no countable exhaustive chains of non-open subgroups in $G$.) As a matter of fact, the automatic continuity property also implies that every isometric action of $G$ on a separable metric space is continuous, and there is only one Polish group topology on $G$.

There is an extensive list of Polish groups known to have ample generics. These include the automorphism groups of the following countable structures: the random graph (Hrushovski \cite{Hr}), the free group on countably many generators (Bryant and Evans \cite{BrEv}), arithmetically saturated models of true arithmetic (Schmerl \cite{Sch}) or the automorphism group of the rational Urysohn space (Solecki \cite{So}.) Also, Malicki \cite{Ma2} characterized all Polish ultrametric spaces whose isometry group contains an open subgroup with ample generics. Moreover, very recently, the first examples Polish groups with ample generics that are not isomorphic to the automorphism group of some countable structure have been found (see \cite{Ma3}, and \cite{KaMa}.) 

On the other hand, many classical and frequently studied Polish groups do not have ample generics. For instance, it has been proved that the isometry group $\Iso(\UU)$ of the Urysohn space $\UU$ (see \cite{GlWe}), the automorphism group $\Aut(\MA)$ of the Lebesgue measure algebra $\MA$  (see \cite{Ju}), and the unitary group $U(\ell_2)$ of the Hilbert space $\ell_2$ (see \cite{GlWe}) have meager conjugacy classes. However, further investigations revealed that they share some properties that are consequences of the existence of ample generics: $U(\ell_2)$ has the automatic continuity property (see \cite{Ts}) and uncountable cofinality (see \cite{RiRo}), while $\Aut(\MA)$ has the automatic continuity property (see \cite{BeBeMe}), and the small index property (see \cite{Ma1}.)


Recently, Sabok \cite{Sa} proposed an approach that sheds new light on these results.  Using model-theoretic methods, he showed that a (complete, homogeneous) metric structure $X$ (understood, essentially, as in \cite{BeBeHeUs}) that satisfies some additional assumptions, always contains countable substructures $Y$ that are elementary submodels of $X$, and whose automorphism groups have ample generics. Moreover, $Y$ behave nicely enough to make a connection between the properties of $\Aut(Y)$, and the properties of $\Aut(X)$. This led him to formulating a criterion for a metric structure $X$ that implies that $\Aut(X)$ has the automatic continuity property. Subsequently, he showed that $\UU$, $\MA$, and $\ell_2$ meet this criterion. He also found a condition implying that every homomorphism of $U(\ell_2)$ into a separable group with a left-invariant, complete metric is trivial (this has been first showed in \cite{Ts}.)

In this paper, we follow Sabok's approach but we propose a different, and quite simple criterion (see Definition \ref{de:SI}) that implies \emph{all} the main consequences of the existence of ample generics: the automatic continuity property, the small index property, and uncountable cofinality for non-open subgroups. We also verify it for $\UU$, $\MA$, and $\ell_2$, thus proving that the groups $\Iso(\UU)$, $\Aut(\MA)$, and $U(\ell_2)$ satisfy these properties. As a matter of fact, since all these groups are connected, they also have uncountable cofinality. 

We also formulate a new condition for a metric structure $X$ (see Definition \ref{de:IHS}) implying that every homomorphism of $\Aut(X)$ into a separable group with a left-invariant, complete metric is trivial, and we verify it for $\UU$, and $\ell_2$.
 
\section{Basic notions}
\paragraph{\textbf{Groups.}} Let $G$ be a topological group. We say that $G$ has the \emph{automatic continuity property} if every homomorphism $\Phi: G \rightarrow K$ into a separable group $K$ is continuous. We say that $H \leq G$ has \emph{small index} if $[G:H]<2^\omega$. We say that $G$ has the \emph{small index property} if every subgroup of $G$ with small index is open in $G$. Finally, $G$ has \emph{uncountable cofinality (for non-open subgroups)} if there are no countably infinite chains $G_1 \leq G_2 \leq \ldots \leq G$ of (non-open) subgroups of $G$ such that $G=\bigcup_k G_k$.

A subset $W \subseteq G$ is called \emph{countably syndetic} if there exist $g_k \in G$, $k \in \NN$, such that $G= \bigcup_k g_kW$. We say that G is \emph{Steinhaus} if there is $k>0$ such that for any symmetric countably syndetic set $W \subseteq G$, $W^k$ contains an open neighborhood in $G$.

\paragraph{\textbf{Metric structures.}} The following framework has been introduced in \cite{BeBeHeUs}, and \cite{Sa}, and in these papers it is discussed much more thoroughly than here.
A \emph{metric structure} is a tuple $(X, d, f_1, f_2, \ldots)$, where $(X,d)$ is a separable metric space and $f_1, f_2,\ldots$ are either closed subsets of $X^n$ (relations) for some $n\in \NN$, or continuous functions from $X^n$ to $X$, or to the reals $\RR$, for some $n \in \NN$. Thus, a metric structure is a two-sorted structure with the second sort being a subset of the real line $\RR$, and $d$ regarded as a function with arguments in the first sort, and values in the second sort. If $(X,d)$ is a complete metric space, we say that $(X, d, f_1, f_2, \ldots)$ is a \emph{complete metric structure}.
%

Given a metric structure $X$ we write $\Aut(X)$ for the group of all automorphisms of $X$ (i.e. bijections of the first sort of $X$ which preserve both the metric and each $f_n$.) $\Aut(X)$ is always endowed with the topology of pointwise convergence, so, if $X$ is complete, $\Aut(X)$ is a Polish group. 

For $G=\Aut(X)$, we denote by $G_{\bar{a}}$ the pointwise stabilizer of $\bar{a}$ in $G$. Also, for $W \subseteq G$, $W[\bar{a}]=\{ w(\bar{a}) \in X^n: w \in W \}$.

Terms and formulas in metric structures are defined as in the first-order logic, the only differences are that quantification is allowed only over the first sort, and that elements of the second sort are regarded as constants (e.g. expressions of the form $d(x,y)=r$, where $r$ is a real number, are quantifier-free formulas.) The truth value of a first-order sentence in a metric structure is also defined as in the classical setting (it is either $0$ or $1$.) The notions of a quantifier free type $p(\bar{x})$, the quantifier free type $\qftp{\bar{a}}$ of a tuple $\bar{a}$ (in the first sort), and the the quantifier free type $\qftp{\bar{a}/\bar{b}}$ of a tuple $\bar{a}$ over a tuple $\bar{b}$ are defined in the standard way. 

For an $n$-tuple $\bar{a}$ of elements in $X$, an $n$-tuple $\bar{x}$ of variables in the first sort, and $\epsilon>0$, we say that a quantifier-free type $p(\bar{x}/\bar{a})$ is an $\epsilon$-\emph{quantifier-free type over} $\bar{a}$ if $\qftp{\bar{a}} \subseteq p$ and $d(a^i, x^i) = d_i$ belongs to $p(\bar{x}/\bar{a})$ for each $i \leq  n$ and for some $0 \leq d_i < \epsilon$.

Given a metric structure $X$, a tuple $\bar{a}$ in $X$ with $p = \qftp{\bar{a}}$, and $\epsilon>0$, say that a subset $Y \subseteq p(X)$ is \emph{relatively} $\epsilon$-\emph{saturated} over $\bar{a}$ if every $\epsilon$-quantifier-free type over $\bar{a}$ which is realized in $X$, is also realized in $Y$. 

We say that a metric structure $X$ is \emph{homogeneous} if every partial isomorphism between finitely generated substructures of $X$ extends to an automorphism of $X$ (such a structure is also called ultrahomogeneous.) Observe that if $X$ is complete and homogeneous, and $p=\qftp{\bar{a}}$ for an $n$-tuple $\bar{a}$ in $X$, then $G[\bar{a}]$ is a $G_\delta$ subset of $X^n$, that is, it is a Polish space, and $G[\bar{a}]=p(X)$. The latter obviously follows from the definition of homogeneity. The former holds because the assumption on closedeness of relations, and continuity of functions implies that, in $X$, the set of realizations of every quantifier-free formula is a $G_\delta$ set. Moreover, a tuple realizes $p$ if and only if it satisfies a countable set of formulas consisting of formulas in $p$ that do not contain a subformula of the form $\tau \neq r$, where $\tau$ is a term, and $r \in \RR$. This is because for every formula $\tau \neq r$ in $p$ there is $s \in \RR$ such that $\tau=s$ is in $p$ as well. 

Let $X$ be a metric structure, $B,C \subseteq X$ be finitely generated substructures. Given a finitely generated substructure $A \subseteq B \cap C$ say that $B$ and $C$ are \emph{independent over} $A$ if for every pair of automorphisms $\phi: B \rightarrow B$, $\psi :C \rightarrow C$ such that $A$ is closed under $\phi$ and $\psi$ and $\phi \upharpoonright A= \psi \upharpoonright A$, the function $\phi \cup \psi$ extends to an automorphism of the substructure generated by $B$ and $C$.

A metric structure $X$ has the \emph{extension property} if for every pair $B,C  \subseteq X$ of finitely generated substructures and a finitely generated substructure $A \subseteq B \cap C$ there is a finitely generated substructure $C' \subseteq X$ with $C'$ which is elementarily equivalent to $C$ over $A$, and is such that $B$, $C'$ are independent over $A$.

Finally, we say that a metric structure $X$ has \emph{locally finite automorphisms} if for every $k \in \NN$, for any finitely generated substructure $Y$ of $X$, and for any partial automorphisms $\phi_1, \ldots, \phi_k$ of $Y$, there is a finitely generated structure $Y'$ of $X$ containing $Y$ such that every $\phi_i$ extends to an automorphism of $Y'$.


\section{Consequences of the existence of ample generics}
We associate with every metric structure $X$ with $G=\Aut(X)$ a pre-defined family of \emph{relevant tuples}, that is, some fixed family $R$ of tuples in $X$ such that for every tuple $\bar{a}$ in $X$ there exists $\bar{b} \in R$ such that $G_{\bar{b}} \leq G_{\bar{a}}$.

\begin{definition} \label{de:SI}
Let $X$ be a metric structure, and let $G=\Aut(X)$. A tuple $\bar{a}$ in $X$ with $p=\qftp{\bar{a}}$ is called \emph{ample} if the set of all the tuples $\bar{b} \in p(X)$ such that $G_{\bar{b}}[\bar{a}]$ is relatively $\delta$-saturated over $\bar{a}$ for some $\delta>0$, is comeager in $p(X)$.
\end{definition}


%
%
The next lemma is a counterpart of \cite[Lemma 6.1]{Sa}, where an analogous statement is proved for symmetric, countably syndetic sets.

\begin{lemma}
\label{th:1}
Let $X$ be a complete, homogeneous metric structure that has locally finite automorphisms and the extension property, and let $G=\Aut(X)$. 
\begin{enumerate}
\item For every subgroup $W \leq G$ with small index there exists a tuple $\bar{a}$ in $X$ such that $G_{\bar{a}} \leq W$,
\item for every countable chain of subgroups $W_1 \leq W_2 \leq \ldots \leq G$ such that $G=\bigcup_k W_k$,  there exists a tuple $\bar{a}$ in $X$, and $k \in \NN$ such that $G_{\bar{a}} \leq W_k$.
\end{enumerate}
\end{lemma}

\begin{proof}
To prove Point (1), fix a subgroup $W \leq G$ with small index, and suppose that the statement of the lemma does not hold for $W$. Then for every tuple $\bar{a}$ in $X$ there exists $f_{\bar{a}} \in G_{\bar{a}}$ such that $f_{\bar{a}} \not \in W$. By \cite[Corollary 4.4]{Sa}, there exists a countable, dense elementary submodel $X_0 \prec X$ which is closed under $f_{\bar{a}}$ for tuples $\bar{a}$ in $X_0$. Since $X_0$ is dense in $X$, every automorphism of $X_0$ can be uniquely extended to an automorphism of $X$. Thus, $H=\Aut(X_0)$ can be regarded as a subgroup of $G$.

Now $V=W \cap H$ has small index in $H$. By \cite[Lemma 5.1]{Sa}, and properties of groups with ample generics, there exists a tuple $\bar{b}$ in $X_0$ such that $H_{\bar{b}} \leq V \leq W$. But this is a contradiction because $X_0$ is closed under $f_{\bar{b}}$, that is,  $f_{\bar{b}} \in H_{\bar{b}} \setminus W$. Therefore there exists a tuple $\bar{a}$ such that $G_{\bar{a}} \leq W$.

In order to prove Point (2), fix a countable chain of subgroups $W_1 \leq W_2 \leq \ldots \leq G$ such that $G=\bigcup_k W_k$.  To show that there exist $\bar{a}$, and $k \in \NN$ such that $G_{\bar{a}} \leq W_k$, assume the contrary, and for every tuple $\bar{a}$ and $k \in \NN$ fix $f_{\bar{a},k}$ such that  $f_{\bar{a},k} \in G_{\bar{a}}$, and $f_{\bar{a},k} \not \in W_k$. Then, using \cite[Corollary 4.4]{Sa} again, fix a countable, dense $X_0 \prec X$ closed under $f_{\bar{a},k}$ for tuples $\bar{a}$ in $X_0$ and $k \in \NN$, and let $H=\Aut(X_0)$ be regarded as a subgroup of $G$ (formally, \cite[Corollary 4.4]{Sa} does not allow for the extra index $k$ but its proof, applied verbatim, shows that it is possible to find such $X_0$.) Then $V_k=W_k \cap H$ is a chain of subgroups of $H$ such that $\bigcup_k V_k=H$. By \cite[Lemma 5.1]{Sa}, there exists $k \in \NN$ such that either $V_k$ is open in $H$ (in the product topology) or there exists $k$ such that $H=V_k$. In any case, there exists a tuple $\bar{b}$ in $X_0$ such that $H_{\bar{b}} \leq W_k \cap H$. As before, this leads to a contradiction, so there must exist $\bar{a}$, and $k$ such that $G_{\bar{a}} \leq W_k$. 
\end{proof}

\begin{lemma} \label{le:2in1}
Let $X$ be a homogeneous metric structure, and let $G=\Aut(X)$. Suppose that $W \subseteq G$ is a symmetric set such that there exists a tuple $\bar{a}$ in $X$, and $\delta>0$ such that $G_{\bar{a}} \subseteq W$, and $W[\bar{a}]$ is relatively $\delta$-saturated over $\bar{a}$. 
Then $W^3$ contains an open neighborhood in $G$. 
\end{lemma}

\begin{proof}
Fix $\bar{a}$ and $\delta$ as in the statement of the lemma. Let $V \subseteq G$ be an open set such that $d(g(\bar{a}),\bar{a}) < \delta$ for $g \in V$. Fix $g \in V$. Because $W[\bar{a}]$ is relatively $\delta$-saturated over $\bar{a}$, there exists a tuple $\bar{c} \in W[\bar{a}]$ such that
\[ \qftp{\bar{c}/\bar{a}}=\qftp{g(\bar{a})/\bar{a}}. \]

Because $X$ is homogeneous, there exists $w_1 \in G_{\bar{a}} \subseteq W$ such that 
\[ w_1g(\bar{a})=\bar{c}. \]
Fix $w_2 \in W$ such that $w_2(\bar{c})=\bar{a}$. Then
\[ w_2w_1g(\bar{a})=\bar{a}, \]
that is,
\[ g \in w_1^{-1}w_2^{-1}G_{\bar{a}} \subseteq W^3. \]

As $g \in V$ was arbitrary, we have that $ V \subseteq W^3$. 
\end{proof}


\begin{theorem} \label{th:Main}
Let $X$ be a complete, homogeneous metric structure that has locally finite automorphisms and the extension property, and let $G=\Aut(X)$. If every relevant tuple in $X$ is ample, then 
\begin{enumerate}
\item $G$ the small index property,
\item $G$ has the automatic continuity property,
\item $G$ has uncountable cofinality for non-open subgroups.
\end{enumerate}
\end{theorem}

\begin{proof}
We prove Point (1). Fix a subgroup $W \leq G$ with small index. By Lemma \ref{th:1}, there exists an ample relevant tuple $\bar{a}$ in $X$, with $p=\qftp{\bar{a}}$, such that $G_{\bar{a}} \subseteq W$. 
As $W$ is non-meager in $G$ (see remarks on page 5 of \cite{Be}), and $p(X)$ is a Polish space, the Effros theorem (see \cite[Theorem 3.2.4]{Gao}) implies that  $W[\bar{a}]$ is non-meager in $p(X)$. Since $\bar{a}$ is ample, there exists $w \in W$, and $\delta>0$ such that $G_{w(\bar{a})}[\bar{a}]$ is relatively $\delta$-saturated over $\bar{a}$. But $wG_{\bar{a}}w^{-1}=G_{w(\bar{a})}$, so $G_{w(\bar{a})} \leq W$. By Lemma \ref{le:2in1}, $W$ contains an open neighbourhood in $G$, that is, $W$ is open in $G$.

In order to prove Point (2), fix a symmetric, and countably syndetic $W \subseteq G$. Since $W$ is non-meager in $G$, arguing exactly as above, we can show that $W^{36}$ contains an open neighbourhood in $G$. As $W$ was arbitrary, $G$ is Steinhaus, and so it has the automatic continuity property (see \cite[Proposition 2]{RoSo}.)

Point (3) can be proved in the same way because for every chain $W_0 \leq W_1 \leq \ldots \leq G$ such that $\bigcup_k W_k=G$, almost all $W_k$ are non-meager in $G$.
\end{proof}

\section{Triviality of homomorphisms}

We denote by $\SI$ the infinite symmetric group with the product topology.

\begin{definition} \label{de:IHS}
Let $X$ be a homogeneous metric structure, and let $G=\Aut(X)$. A tuple $\bar{a}$ in $X$, with $p=\qftp{\bar{a}}$, is called \emph{homogeneously isolated} if there exist $X_k \subseteq X$, $\bar{a}_k \in p(X)$, $k \in \NN$, such that for every $k$
\begin{enumerate}
\item $X_k$ is a homogeneous substructure of $X$,
\item for all $\phi_l \in \Aut(X_l)$, $l \in \NN$, there exists $\phi \in \Aut(X)$ that extends each $\phi_l$,
\item there exists a homomorphic embedding $e:\SI \rightarrow G$ such that for every $\sigma \in \SI$ and $k,l \in \NN$
\[ \sigma(k)=l \Leftrightarrow e(\sigma)[X_k]=X_l. \]
\end{enumerate}
\end{definition}

\begin{theorem} \label{th:Triv}
Let $X$ be a complete, homogeneous metric structure that has locally finite automorphisms and the extension property, and let $G=\Aut(X)$. If every tuple $\bar{a}$ in $X$ is homogeneously isolated, then every homomorphism $\Phi: G \rightarrow K$, where $K$ is a separable group with a left-invariant, complete metric, is trivial.
\end{theorem}

\begin{proof}
Let $\Phi: G \rightarrow K$ be a homomorphism into a separable group $K$ with a left-invariant, complete metric. Fix a symmetric neighborhood of the identity $U$  in $K$. Fix a symmetric neighborhood $V$ such that $V^{12} \subseteq U$. Let $W=\Phi^{-1}(V)$. Then $W$ is symmetric and countably syndetic in $G$ so, by \cite[Lemma 6.1]{Sa}, there exists a tuple $\bar{a}$ in $X$, with $q=\qftp{\bar{a}}$, such that $G_{\bar{a}} \subseteq W^{10}$. Fix $g \in G$, and $X_k$,  $\bar{\alpha}_k$, $k \in \NN$, witnessing that $(\bar{a}, g(\bar{a}))$ with $p=\qftp{(\bar{a}, g(\bar{a}))}$ is homogeneously isolated.

Without loss of generality we can assume that the embedding $e$ from Condition (3) is the identity. Fix $f \in G$ such that $f(\bar{a}, g(\bar{a}))$ is a tuple in $X_1$ (this is possible because $\bar{\alpha}_1 \in p(X_1)$, and $X$ is homogeneous), and fix $\sigma_k \in \SI$, $k \in \NN$, such that $f_k(\bar{a},g(\bar{a})) \in X_k$ for $f_k=\sigma_k f$. 

\emph{Claim:} There exists $k \in \NN$ such that $f_k W^2 f_k^{-1}$ is full on $q(X_k)$.

To show the claim, recall that $\SI$ has a unique non-trivial separable group topology (see \cite[Theorem 1.11]{KeRo}.) Also, it is well known that $\SI$ has no non-trivial closed normal subgroups. As $\SI$ has ample generics (see \cite{KeRo}), and thus $\Phi$ restricted to $\SI$ is continuous, $\Phi[\SI]$ is either trivial or it is an isomorphic copy of $\SI$ in $K$. But $K$ has a left-invariant, complete metric, while it is well known that $\SI$ does not admit such a metric. Therefore $\Phi[\SI]$ must be the trivial subgroup of $K$.

Now we prove that there exist $b_k \in G$, $k \in \NN$, such that $G=\bigcup_k b_k W f_k^{-1}$. Actually, it is easy to see that it suffices to find $a_k \in K$, $k \in \NN$, such that $K=\bigcup_k a_k V \Phi(f_k^{-1})$. But 
\[ \Phi(f_k^{-1})=\Phi( f^{-1}\sigma^{-1}_k)=\Phi(f^{-1})\Phi(\sigma^{-1}_k) =\Phi(f^{-1}) \]
for every $k$, so we can find such $a_k$ because $V$ is open, and $K$ is separable.

Observe that there exists $k \in \NN$ such that $b_k W f_k^{-1}$ is full on $q(X_k)$. If not, then for each $k$ there is $\phi_k \in \Aut(X_k)$ such that $\phi_k \neq g \upharpoonright X_k$ for every $ g \in b_kW  f_k^{-1}$. Fix an automorphism $\phi \in \Aut(X)$ such that $\phi \upharpoonright X_k =\phi_k$ for $k \in \NN$. Then $\phi \not \in \bigcup_k b_kW f_k^{-1}$, which is a contradiction. Now, if $b_k W f_k^{-1}$ is full on $q(X_k)$, then so is $f_k W^2 f_k^{-1} = (b_kW f_k^{-1})^{-1}(b_kW f_k^{-1})$ as for every tuple in $q(X_k)$ the set $b_kW_k f_k^{-1}$ contains an element that acts trivially on it. This finishes the proof of the claim.

Fix $k$ such that $f_k W^2 f_k^{-1}$ is full on $q(X_k)$. Fix $w \in W^2$ such that
\[ f_kw f_k^{-1} f_kg(\bar{a})=f_k(\bar{a}), \]
that is,
\[ f_k wgf_k^{-1}(\bar{a}) \in G_{f_k(\bar{a})},  \]
and
\[ wg \in G_{\bar{a}}. \]
Finally,
\[ g \in w^{-1}G_{\bar{a}} \leq W^{12}, \]
and
\[ \Phi(g) \in V^{12} \subseteq U. \]

Since $U$ and $g$ were arbitrary, this shows that $\Phi$ is trivial
\end{proof}

\section{The Urysohn space}

For every Polish metric space $X$ there exists a universal and homogeneous Polish metric space $\UU(X)$ such that every isometry of $X$ uniquely extends to an isometry of $\UU(X)$ (see \cite{Ka}.) As a matter of fact, there exists only one, up to isometry, universal and homogeneous Polish metric space, and it is called the \emph{Urysohn space}. We denote it by $\UU$, its metric by $d$, and view $\UU$ as a metric structure $(\UU,d)$. 

It is well known that $\UU$ has the following property which we call here the \emph{universal extension property}: for every finite metric space $B$, and $A \subseteq B$, every embedding $e:A \rightarrow \UU$ can be extended to an embedding of $B$ into $\UU$.

In the Urysohn space, tuples of pairwise distinct points are relevant.

\begin{lemma}
\label{le:SI:U}
Every tuple in $\UU$ consisting of distinct elements is ample.
\end{lemma}

\begin{proof}
In order to simplify notation, we will identify ordered tuples of distinct elements with unordered (but indexed) tuples, that is, sets. Fix an $n$-tuple $\bar{a}=\{a^1, \ldots, a^n\}$ in $\UU$ consisting of distinct elements, and let $p=\qftp{\bar{a}}$.  Let $T$ be the set of all the tuples $\bar{b} \in p(\UU)$ that are disjoint from $\bar{a}$. We show that $T$ witnesses that $\bar{a}$ is ample.

Clearly, $T$ is open and dense, and so comeager. Fix $\bar{b}=\{b^1, \ldots, b^n\} \in T$, and $\delta>0$ such that $d(a^i,b^j)>4\delta$, for all $i, j \leq n$, and $d(a^i,a^j)>4\delta$ for all $i \neq j$. Let $(X,\vartheta)$ be a metric space defined as follows. First of all, find $W^i \subseteq \UU$ with $a^i \in W^i$, $i \leq n$, that are isometric to the sphere in $\UU$ of radius $\delta$ centered at $a^i$, respectively. This can be easily done in the following way.

Put $V^i$ to be the sphere in $\UU$ of radius $\delta$ centered at $a^i$, $i \leq n$, fix an $n$-tuple $\bar{c}=\{c^1, \ldots, c^n \}$ disjoint from $\bar{a}$, and define a metric $\vartheta_0$ on $\bar{a} \cup \bar{c}$ so that $\vartheta_0$ agrees with $d$ on $\bar{a}$, $\vartheta_0(a^i,c^i)=\delta$, and $\vartheta_0(c^i,a^j)=\vartheta_0(c^i,c^j)=d(a^i,a^j)$ for $i,j \leq n$ with $i \neq j$. By the universal extension property, the identity embedding of $\bar{a}$ into $\UU$ can be extended to $(\bar{a} \cup \bar{c}, \vartheta_0)$, so we can assume that actually $c^i \in V^i$. Using homogeneity of $\UU$, find an isometry $\phi_0$ of $\UU$ such that $\phi_0(c^i)=a^i$, and put $W^i=\phi_0[V^i]$, $i \leq n$.

Now put $W=W^1 \cup \ldots \cup W^n$, $X=W \cup \{b^1, \ldots, b^n \}$, and let $\vartheta_1$ be a metric on $X$ such that, when restricted to $W$ or $\{b^1, \ldots, b^n\}$, it agrees with $d$, and is given by
\[ \vartheta_1(w,b^j)=d(a^i,b^j) \]
for every $i,j \leq n$ and $w \in W^i$.

We can assume that $\UU=\UU(X)$. Put $G=\Aut(\UU(X))$. Obviously, each $W^i$ is the sphere in $\UU$ of radius $\delta$ centered at $\phi_0(a^i)$. Fix an $n$-tuple $\bar{c}=\{c^1, \ldots, c^n\} \in p(\UU)$ such that $d(c^i,a^i)<\delta$, $i \leq n$, and fix a tuple $\bar{e}=\{e^1, \ldots, e^n\}$ that is disjoint from $X \cup \bar{c}$. Let $\vartheta_2$ to be a metric on $\bar{a} \cup \phi_0(\bar{a}) \cup \bar{e}$ that agrees with $d$ on $\bar{a} \cup \phi_0(\bar{a})$, and satisfies $\vartheta_2(e^i,e^j)=d(c^i,c^j)$, $\vartheta_2(e^i,a^j)=d(c^i,a^j)$, and $\vartheta_2(e^i,\phi_0(a^j))=d(a^i,\phi_0(a^j))$, $i,j \leq n$. By the universal extension property, we can assume that actually $e^i \in W^i$, $i \leq n$. Moreover, it is easy to verify that there exists $\phi_1 \in G_{\bar{b}}$ such that $\phi_1(\bar{a})=\bar{e}$. As $\bar{c}$ was arbitrary, this shows that $G_{\bar{b}}[\bar{a}]$ is relatively $\delta$-saturated over $\bar{a}$.
\end{proof}

\begin{theorem} \label{th:Ur}
The group $\Iso(\UU)$ has the automatic continuity property, the small index property, and uncountable cofinality.
\end{theorem}

\begin{proof}
The Urysohn space has locally finite automorphisms and the extension property (see \cite{So} and Lemma 8.9 in \cite{Sa}.) Therefore, by Theorem \ref{th:Main}, and Lemma \ref{le:SI:U}, we get that $\Iso(\UU)$ has  the automatic continuity property, the small index property, and uncountable cofinality for non-open subgroups. But $\Iso(\UU)$ is known to be connected, so it does not contain any non-trivial open subgroups. Thus, $\Iso(\UU)$ has uncountable cofinality as well.
\end{proof}

\begin{theorem}
Every tuple in $\UU$  is homogeneously isolated. In particular, every homomorphism $\Phi: \Iso(\UU) \rightarrow K$ into a separable topological group $K$ with a left-invariant, complete metric is trivial.
\end{theorem}

\begin{proof}
Let $\bar{a}$ be a tuple in $\UU$. Fix $x \in \UU$, and define
\[ r=\max \{ d(a^i,a^j): i,j \leq n \}, \, W=\{y \in \UU: d(x,y)=r \}. \]

Let $X_k$, $k \in \NN$, be pairwise disjoint copies of $W$, and define a metric $\rho$ on $X=\bigcup_k X_k$ by putting
\[ \rho(x,y)=2r \]
for $x \in X_k$, $y \in X_l$ with $k \neq l$. We can assume that $\UU=\UU(X)$. Clearly, for all $\phi_k \in \Aut(X_k)$, we have that $\bigcup_k \phi_k \in \Aut(X)$, so there exists a unique $\phi \in \Aut(\UU)$ that extends each $\phi_k $. The remaining conditions for homogeneous isolatedeness are obviously satisfied by $X_k$, $k \in \NN$.
\end{proof}

\section{The measure algebra}

Let $(X,\mathcal{B}, \mu)$ be the standard probability measure space. Let $\mbox{MALG}$ be the measure algebra of sets in $\mathcal{B}$ modulo null sets (even though, formally, elements of $\MA$ are classes of sets, we will refer to them as sets). Then $\MA$ can be endowed with a metric defined by
\[ \rho(A,B)=\mu(A\triangle B), \]
and we treat it as a metric structure together with this metric, the operation of symmetric difference $\Delta$, and the empty set $\emptyset$ as a constant. 

It is well known that  $\Aut(\MA)$ is isomorphic to the group $\Aut(X,\mu)$ of all measure preserving transformations of $X$ with the weak topology. 

In the measure algebra, tuples that are finite partitions of $X$ into positive measure sets are relevant. In \cite{Sa}, the following lemma has been proved (Lemma 9.5):

\begin{lemma} \label{le:Sa:MA}
Let $\bar{a} = (A^1, \ldots,A^n)$ be a partition of $X$ into positive measure sets, and let $p = \qftp{\bar{a}}$. Suppose $C^1,\ldots,C^n$ are such that $C^i \subseteq  A^i$ for each
$i \leq n$, and $\mu(C^1) = \ldots = \mu(C^n) > 0$. Let
\[ M = \{(B^1, \ldots ,B^n) \in p(\MA) : \forall \,  i \neq j \leq  n \  B^i \cap A^j \subseteq C^j \ \wedge  A^i \setminus B^i  \subseteq C^i \}. \]
Then $M$ is relatively $2\mu(C^1)$-saturated over $\bar{a}$.
\end{lemma}

\begin{lemma} \label{le:SI:MA}
Every finite partition of $X$ into positive measure sets is ample.
\end{lemma}

\begin{proof}
Fix a partition $\bar{a} = (A^1, \ldots,A^n)$ of $X$ into positive measure sets, and let $p = \qftp{\bar{a}}$. Let $T$ be the set of all the tuples $\bar{b} \in p(\MA)$ such that some element of $\bar{b}$ has a non-trivial intersection with every $A^i$.  We show that $T$ witnesses that $\bar{a}$ is ample.

Clearly, $T$ is open and dense, and so comeager. Fix $\bar{b}=(B^1, \ldots, B^n) \in T$, and $i_0 \leq n$ such that $B^{i_0}$ has a non-trivial intersection with every $A^i$. Clearly, we can find $W \subseteq B^{i_0}$ such that 
\[ 0<\mu(W \cap A^1)=\mu(W \cap A^i) \]
for every $i \leq n$. Put $G=\Aut(\MA)$. It is easy to see that $M \subseteq G_{\bar{b}}[\bar{a}]$, where $M$ is defined as in Lemma \ref{le:Sa:MA} for $\bar{a}$, and $C^i=W \cap A^i$, $i \leq n$. Therefore $G_{\bar{b}}[\bar{a}]$ is relatively $2\mu(W \cap A^1)$-saturated over $\bar{a}$.
\end{proof}

\begin{theorem}
The group $\Aut(X,\mu)$ has the automatic continuity property, the small index property, and uncountable cofinality.
\end{theorem}

\begin{proof}
The measure algebra has locally finite automorphisms and the extension property (see Lemma 9.1 and Lemma 9.2 in \cite{Sa}). Therefore, by Lemma \ref{le:SI:MA}, and Theorem \ref{th:Main}, we get that $\Aut(X,\mu)$ has the automatic continuity property, the small index property, and uncountable cofinality for non-open subgroups. But $\Aut(X,\mu)$ is known to be connected, so it does not contain any non-trivial open subgroups. Thus, $\Aut(X,\mu)$ has uncountable cofinality as well.
\end{proof}

\section{The Hilbert space}
The orthogonal group $O(\ell_2)$  is the group of all linear transformations of the real Hilbert space that preserve the inner product. Here, the real Hilbert space is treated as the metric structure $H$ with $\ell_2$ as the universe of the first sort, and the second sort being the real line with the field structure (including the inverse function defined on non-zero elements by $x \mapsto  x^{-1}$, and mapping $0$ to $0$, as well as the function $x \mapsto -x$) and constants for the rationals. We also add to the language the constant $0$ for the zero vector, the addition $+$, the inner product function $\langle \cdot , \cdot \rangle:  \ell_2 \times \ell_2 \rightarrow \RR$, and the multiplication by scalars function  $\cdot : \RR \times \ell_2 \rightarrow \ell_2$ (i.e. $(a, v) \mapsto→ a \cdot v$). Clearly, $\Aut(H)$  is isomorphic to $O(\ell_2)$.

The complex Hilbert space is defined as a metric structure analogously to the real Hilbert space. Then $\Aut(H)$ is isomorphic to the unitary group $U(\ell_2)$.

In the Hillbert space, real or complex, orthonormal tuples are relevant. In the next lemma, we consider only the real Hilbert space but the proof is exactly the same for the complex Hilbert space.

\begin{lemma} \label{le:SI:H}
Every orthonormal tuple in $H$ is ample.
\end{lemma}

\begin{proof}

Fix an orthonormal tuple $\bar{a}=(a^1, \ldots, a^n)$ in $H$, and let $p = \qftp{\bar{a}}$. 
Let $T$ be the set of all the tuples $\bar{b} \in p(H)$ such that $\spa(\bar{a},\bar{b})$ has dimension $2n$. We show that $T$ witnesses that $\bar{a}$ is ample.

Clearly, $T$ is open and dense, and so comeager. Fix $\bar{b} \in T$. It is easy to construct an orthonormal set
\[ C=C_1 \dot\cup \ldots \dot\cup C_n \]
in $(\spa(\bar{b}))^\perp$ such that each $C_i$ has size $n$, the projections $\pi_{c}(a^i)$ of $a^i$ on $c$ are trivial for every $i,j \leq n$ with $i \neq j$, $c \in C_j$, and the projection $\pi_c(a^i)$ is non-trivial for every $c \in C_i$. The crucial observation here is that, for each $i$, if $A$ is a finitely dimensional subspace of $H$ such that $a^i \not \in A$, then $A^\perp \not \subseteq (\spa(a^i))^\perp$, and there exists $a \in A^\perp \setminus (\spa(a^i))^\perp$ such that $a^i \not \in \spa(A,a)$.

Also, we fix an orthonormal set
\[ D=D_1 \dot\cup \ldots \dot\cup D_n \]
in $(\spa(\bar{a}, \bar{b},C))^\perp$ such that each $D_i$ has size $n$.


%
%


Now, using $C$ and $D$, we will find $\delta>0$ such that $G_{\bar{b}}[\bar{a}]$ is $\delta$-relatively saturated over $\bar{a}$. Let $C_i=\{c_{i,j}\}_{j \leq n}$, $D_i=\{d_{i,j}\}_{j \leq n}$, $i \leq n$, and let 

\[ \delta=(\min \{ \left\| \pi_{c_{i,j}}(a^i) \right\|: i,j \leq n \})^2. \]

Observe that each $a^i$ can be independently rotated  in each of the planes $\spa(c_{i,j}, d_{i,j})$, $j \leq n$, using elements of $G_{\bar{b}}$, to modify the values of coordinates $c_{i,j}$ between $0$ and $\delta$, while keeping the resulting tuple orthonormal.  As $\pi_{d_{i,j}}(a^i)$ is trivial for all $i,j \leq n$, and $\pi_{c_{i',j}}(a^i)$ is trivial if $i \neq i'$, this means that for all $0 \leq \delta_{i,j} < \delta$, $i,j \leq n$, there exists $(c^1, \ldots, c^n) \in G_{\bar{b}}[\bar{a}]$ such that
%
%
\[ \langle a^i,c^j \rangle=1-\delta_{i,j}. \]
for $i,j \leq n$. Thus, $G_{\bar{b}}[\bar{a}]$ is $\delta$-relatively saturated over $\bar{a}$.
\end{proof}


\begin{theorem}
The groups $O(\ell_2)$ and $U(\ell_2)$ have the automatic continuity property, the small index property, and uncountable cofinality.
\end{theorem}

\begin{proof}
The Hilbert space has locally finite automorphisms and the extension property  (see Lemma 10.12 and Corollary 10.4 in \cite{Sa}). Therefore, by Lemma \ref{le:SI:H}, and Theorem \ref{th:Main}, we get that $O(\ell_2)$ and $U(\ell_2)$ have the automatic continuity property, the small index property, and uncountable cofinality for non-open subgroups. But $O(\ell_2)$ and $U(\ell_2)$ are known to be connected, so they do not contain any non-trivial open subgroups. Thus, $O(\ell_2)$ and $U(\ell_2)$ have uncountable cofinality as well.
\end{proof}

It has been proved in \cite{Ts}, and later in \cite{Sa} (see Theorem 7.3 and Lemma 10.14) that every homomorphism $\Phi: O(\ell_2) \rightarrow K$ or $\Phi: U(\ell_2) \rightarrow K$ into a separable topological group $K$ with a  left-invariant, complete metric is trivial. We provide an alternative, very short proof of this result.

\begin{theorem}
Every tuple in the Hilbert space is homogeneously isolated. In particular, every homomorphism $\Phi: O(\ell_2) \rightarrow K$ or $\Phi: U(\ell_2) \rightarrow K$ into a separable topological group $K$ with a left-invariant, complete metric is trivial.
\end{theorem}

\begin{proof}
Let $X_k \subseteq H$, $k \in \NN$, be pairwise orthogonal copies of the Hilbert space. It is straightforward to check that $X_k$, $k \in \NN$, witness that every tuple in $H$ is homogeneously isolated.
\end{proof}

\end{document}